\documentclass[preprint,11pt]{elsarticle}

\usepackage{amssymb,latexsym,amsmath}     % Standard packages
\usepackage{epsf}
\usepackage{array}
\usepackage{amssymb,amscd}

\usepackage{relsize}%
\usepackage{calc}
\usepackage[mathscr]{eucal}
\usepackage{hhline}
\usepackage{tikz}
\usepackage{multirow}
\usetikzlibrary{arrows.meta}
\usepackage{amssymb,amsmath,amstext,amsgen,amsfonts,amsthm,mathrsfs,verbatim,url,hyperref,mathdots}
\usepackage{enumitem}
\usepackage{pgfplots}
\pgfplotsset{compat=1.15}
\usepackage{mathrsfs}
\usetikzlibrary{arrows}

\usepackage{arydshln}
\newlist{inparaenum}{enumerate}{2}% allow two levels of nesting in an enumerate-like environment
\setlist[inparaenum]{nosep}% compact spacing for all nesting levels
\setlist[inparaenum,1]{label=\bfseries\alph*.}% labels for top level

\newtheorem{theorem}{Theorem}[section]

\newtheorem{example}[theorem]{Example}

\newtheorem{lemma}[theorem]{Lemma}
\newtheorem{Conjecture}[theorem]{Conjecture}
\newtheorem{remark}[theorem]{Remark}
\newtheorem{proposition}[theorem]{Proposition}
\newtheorem{corollary}[theorem]{Corollary}

\def\comment#1{}

\def\invddots{\mathinner{\mskip1mu\raise1pt\vbox{\kern7pt\hbox{.}}\mskip2mu
		\raise4pt\hbox{.}\mskip2mu\raise7pt\hbox{.}\mskip1mu}}
\journal{TBD}

\begin{document}

\begin{frontmatter}

%% Title, authors and addresses

%% use the tnoteref command within \title for footnotes;
%% use the tnotetext command for theassociated footnote;
%% use the fnref command within \author or \address for footnotes;
%% use the fntext command for theassociated footnote;
%% use the corref command within \author for corresponding author footnotes;
%% use the cortext command for theassociated footnote;
%% use the ead command for the email address,
%% and the form \ead[url] for the home page:
%% \title{Title\tnoteref{label1}}
%% \tnotetext[label1]{}
%% \author{Name\corref{cor1}\fnref{label2}}
%% \ead{email address}
%% \ead[url]{home page}
%% \fntext[label2]{}
%% \cortext[cor1]{}
%% \address{Address\fnref{label3}}
%% \fntext[label3]{}

\title{Backward Stability of the Schur Decomposition under Small Perturbation}

%% use optional labels to link authors explicitly to addresses:
 \author[label1]{A. Minenkova, E. Nitch-Griffin, V. Olshevsky}
 \address[label1]{University of Connecticut}
%% \address[label2]{}

%\author[label2]{}

%\address[label2]{
%	{Department of Mathematics,  University of Connecticut, Storrs CT 06269-3009,  USA.}}

\begin{abstract}
%% Text of abstract
In the present paper, we  show the backward stability of the Schur decomposition for a given matrix under small perturbation.  
\end{abstract}

\begin{keyword} perturbation theory  \sep Schur decomposition \sep unitary Hessenberg matrices \sep the Gohberg-Kaashoek numbers \sep invariant subspaces \sep gaps.
%% keywords here, in the form: keyword \sep keyword

%% PACS codes here, in the form: \PACS code \sep code

%% MSC codes here, in the form: \MSC code \sep code
%% or \MSC[2008] code \sep code (2000 is the default)

\end{keyword}

\end{frontmatter}

%% \linenumbers

%% main text
\section{Introduction}

 Lipschitz-H\"{o}lder stability was investigated for several canonical forms like Jordan, flipped-orthogonal, flipped-orthogonal conjugare symmetrical, real canonical forms~\cite{BOP,DMO,GLR86,K66,KGMP03,MP80,SS90,W65}, but it was never addressed for the Schur decomposition.

We begin by recalling some classical results.

\subsection{Eigenvalues' Stability}

The first question to consider is what happens to the eigenvalues of a given matrix under small perturbation. In general it might not be true but if we impose additional restrictions on eigenvalues of matrices then we have Lipschitz stability of the eigenvalues. 

The following result can be found in \cite{BOP}.
\begin{proposition} \label{dimeq}
		Let $A_0$ be an $n\times n$ matrix and $\{\lambda_1,\dots,\lambda_n\}$ be its eigenvalues, and $A$ being its perturbation with $\|A-A_0\|<\varepsilon$ for sufficiently small $\varepsilon$ depending on $A_0$ and the eigenvalues $\mu_j$. If the number of eigenvalues of $A_0$ is the same as of $A$, then there is a certain ordering of them such that for some positive $K=K(A_0)$ 
	$$|\mu_i-\lambda_i|\leq K\|A-A_0\|,\quad i=1,2,\dots,|\sigma(A_0)|.$$
\end{proposition}

 For the general case of the eigenvalues stability we have the following result (see~\cite[Appendix K]{Os73}). 
\begin{proposition}\label{eigstab}
	Let $A_0$ be an $n\times n$ matrix and $\{\lambda_1,\dots,\lambda_n\}$ be its eigenvalues. Then, there is an ordering of $\lambda_j$'s that for every $A$ with $\|A-A_0\|<\varepsilon$ for sufficiently small $\varepsilon$ depending on $A_0$ there is an ordering of its eigenvalues $\mu_j$'s and a positive constant $K=K(A_0)$ such that
	\begin{equation}\label{holder}
	|\mu_j-\lambda_j|\leq K\|A-A_0\|^{1/n}.
	\end{equation}\end{proposition}

This type of bounds is called H\"older because of the power $1/n$ for the matrix norm.

 The following example shows that the power $1/n$ in~\eqref{holder} cannot be relaxed and, in general, we can hope only for a H\"older type bound.

 \begin{example}\label{ex1}
 	Consider the following matrices $A_0,A\in \mathbb{C}^{2\times 2}$.
 $$A_0=\begin{bmatrix}
 0&0\\
 1&0
 \end{bmatrix}\text{ and }\ 
 A=\begin{bmatrix}
 0&\epsilon\\
 1&0
 \end{bmatrix}.$$
 \end{example}
 Note that $\|A-A_0\|=\epsilon$. Moreover, $\sigma(A_0)=\{0\}$ and $\sigma(A)=\{\pm\sqrt{\epsilon}\}$. It is easy to see that in this case we have
 $$|0\mp\sqrt{\epsilon}|=\epsilon^{1/2}=\|A-A_0\|^{1/2}.$$
 This example can be easily modified for $n\times n$-matrices.
 
 \subsection{Backward Stability of the Schur Decomposition}
 
 	Every $n\times n$-matrix $A$ is unitary similar to an upper triangular matrix $T$, i.e, $A=UTU^*$ where $U$ is unitary. This triangular matrix $T$ is called a \textit{Schur Triangular form} and the factorization is called the \textit{Schur Decomposition}.
 	
  Note that diagonal entries of $T$ are the eigenvalues of $A$. That is why the eigenvalues stability results give us the confidence to consider stability of the Schur decomposition.

	But what kind of stability can we have for the Schur canonical form? We start by considering the following type of result.
\begin{Conjecture}[\textbf{Forward Stability}]\label{conj}
	Let $A_0 = U_0 T_0U^*_0\in\mathbb{C}^{n\times n}$ where $U_0$ is unitary and $T_0$ is upper triangular. Then, there
	exist constants $K,\epsilon>0$ (depending on $A_0$ only) such that for all $A$ with $\|A-A_0\|<\epsilon$ there
	exists a factorization $U TU^*$ of $A$ such that
	\begin{equation*}
	\|U-U_0\|+\|T-T_0\|\leq K\|A-A_0\|^{1/n}
	\end{equation*}
\end{Conjecture}
We call this property forward stability of the Schur form. 

As the following example shows this conjecture is not valid in the form stated.

\begin{example}
	Consider the following matrix and its perturbation,
$$A_0=\begin{bmatrix}
2&0\\
0&2
\end{bmatrix}\text{ and }
A=\begin{bmatrix}
2&0\\\epsilon&2
\end{bmatrix}.$$
\end{example}

Let us consider the following Schur factorization of $A_0$ 
$$A_0=U_0T_0U_0^*=\begin{bmatrix}
1&0\\0&1
\end{bmatrix}\begin{bmatrix}
2&0\\
0&2
\end{bmatrix}\begin{bmatrix}
1&0\\0&1
\end{bmatrix}$$
and the Schur factorization of $A$ 
$$
A=UTU^*=\begin{bmatrix}
0&1\\
1&0
\end{bmatrix}
\begin{bmatrix}
2&\epsilon\\
0&2
\end{bmatrix} 
\begin{bmatrix}
0&1\\
1&0
\end{bmatrix}.$$

Since the first column of $U$ has to be an eigenvector of $A$, and the latter is essentially unique, the matrix $U$ is essentially unique as well.
Hence, the distance $\|U-U_0\|$ is quite large. Since $U$ was the only possible choice
for triangulating $A$, we can conclude that our Conjecture~\ref{conj} above is false in general.
 Although forward stability results have
been obtained for other canonical forms (see \cite{BOP,DMO}), in  case of the Schur canonical form we cannot have the stability mentioned in Conjecture~\ref{conj}. The next statement shows us why.

\begin{theorem}[Different Gohberg-Kaashoek Numbers]\label{GKnonstable} Let us fix matrix $A_0$ and its fixed Schur decomposition  $A_0=U_0T_0U_0^*$.
There exists $K>0$ such that in any neighborhood of $A_0$, i.e. $\{A:\|A-A_0\|<\varepsilon\}$ for any $\varepsilon>0$,
\begin{equation}\label{nonstable}
 \underset{A}{\sup}\ \underset{U,T}{\inf}\ {\|U-U_0\|+\|T-T_0\|}>M>0,
\end{equation} where the supremum is taken over all $A$ in this neighborhood having different Gohberg-Kaashoek numbers from $A_0$ and the infimum is taken over all their Schur factorizations $A=UTU^*$.
\end{theorem}

%Before giving their formal definition, we present the following result that shows that "the same GK numbers" is a criterion for stability.
 
%If we narrow down the class of perturbations then we hope the forward stability result.

%\begin{Conjecture}[Same Gohberg-Kaashoek Numbers]
%	Let $A_0$ be given. Any Schur canonical form is forward H\"older stable in the class of matrices $A$ having the same GK numbers as $A_0$. This means that there	exists a constant $K>0$ $($depending on $A_0$ only$)$ such that for all $A$ with the same GK numbers as $A_0$ and $\|A-A_0\|<\varepsilon$ there exists a Schur factorization $U TU^*$ of $A$ such that the following inequality holds.	\begin{equation*}
%	\|U-U_0\|+\|T-T_0\|\leq K\|A-A_0\|^{1/n}.
%	\end{equation*} 
%\end{Conjecture}

\subsubsection{Gohberg-Kaashoek Numbers}
Theorem~\ref{GKnonstable} uses Gohberg-Kaashoek (GK) numbers. Let us introduce these numbers now (see \cite{GK78,O89,MO} for details).

	Let $A\in\mathbb{C}^{n\times n}$, $\sigma(A)$ be the set of all its eigenvalues, and $m_1(A,\lambda)\le m_2(A,\lambda)\le\dots\le m_t(A,\lambda)$ be the sizes of all blocks corresponding to $\lambda\in\sigma(A)$ in the Jordan form of $A$. We set $m_i(A,\lambda)=0$ ($i=t+1,\dots,n$) for convenience. The numbers
	$$m_i(A)=\sum_{\lambda\in\sigma(A)}m_i(A,\lambda)$$
	are called \textit{the Gohberg-Kaashoek numbers}.

We can actually prove a more general result, for this we need to define the dual Gohberg-Kaashoek numbers.

Let $m=\left[\begin{smallmatrix}
m_1\\m_2\\ \vdots\\m_n
\end{smallmatrix}\right]$ be a vector with integer entries such that $m_i\geq m_{i+1}$ for $i=1,\dots,n-1$. The vector $k=\left[\begin{smallmatrix}
k_1\\k_2\\ \vdots\\k_n
\end{smallmatrix}\right] $ with
$
k_i=\underset{1\leq l\leq n}{\max}\{l:m_l\geq i\}
$  is called {\textit{dual}} to $m$.

In terms of the Gohberg-Kaashoek numbers $m_j$'s it means that if we have $$A=\left[\begin{array}{cccc|ccc|cc}
\lambda&1&0&0&&&&&\\
0&\lambda&1&0&&&&&\\
0&0&\lambda&1&&&&&\\
0&0&0&\lambda&&&&&\\ \hline
&&&&\lambda&1&0&&\\
&&&&0&\lambda&1&&\\
&&&&0&0&\lambda&&\\\hline
&&&&&&&\lambda&1\\ 
&&&&&&&0&\lambda\\ 
\end{array}\right]$$
then we can put the Jordan chains corresponding to $\lambda$ in the following order.
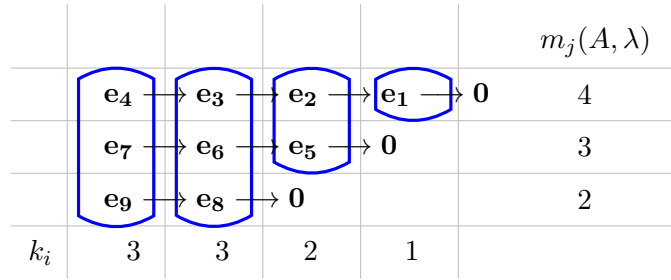
\begin{center}	
	\begin{tikzpicture}
	
	\draw[lightgray] (-1.15,-.65) to (-1.15,3);
	\draw[lightgray] (.15,-.65) to (.15,3);
	\draw[lightgray] (1.45,-.65) to (1.45,3);
	\draw[lightgray] (2.75,-.65) to (2.75,3);
	\draw[lightgray] (4.05,-.65) to (4.05,3);
	%%%%%%%%%%%%%%%%%%%%%%%%%%%%%%%%%%%%%%%%%5
	\draw[lightgray] (-1.9,1.45) to (7,1.45);
	\draw[lightgray] (-1.9,2.15) to (7,2.15);
	\draw[lightgray] (-1.9,.75) to (7,.75);
	\draw[lightgray] (-1.9,.05) to (7,.05);
	%%%%%blue1
	\draw[very thick,blue] (-1.01,2) to[bend left] (0.01,2);
	\draw[very thick,blue] (-1.01,0.19) to[bend right] (0.01,0.19);
	\draw[very thick,blue] (-1,0.175) to (-1,2.015);
	\draw[very thick,blue] (0,0.175) to (0,2.015);
	%%%%%blue2
	\draw[very thick,blue] (.299,2) to[bend left] (1.301,2);
	\draw[very thick,blue] (.299,.19) to[bend right] (1.301,.19);
	\draw[very thick,blue] (.3,.175) to (.3,2.015);
	\draw[very thick,blue] (1.3,.175) to (1.3,2.015);
	%%%%%blue3
	\draw[very thick,blue] (1.599,2) to[bend left] (2.601,2);
	\draw[very thick,blue] (1.599,.9) to[bend right] (2.601,.9);
	\draw[very thick,blue] (1.6,.885) to (1.6,2.015);
	\draw[very thick,blue] (2.6,.885) to (2.6,2.015);
	%%%%%blue4
	\draw[very thick,blue] (2.95,2) to[bend left] (3.95,2);
	\draw[very thick,blue] (2.95,1.6) to[bend right] (3.95,1.6);
	\draw[very thick,blue] (2.95,1.585) to (2.95,2.015);
	\draw[very thick,blue] (3.95,1.585) to (3.95,2.015);
	
	%%%%%%%%%%%%%%%%%%%%%%%%%%%%%%%%
	\draw (5,2.5) node[right] {$m_j(A,\lambda)$};
	\draw (5.5,1.8) node[right] {$4$}
	(5.5,1.1) node[right] {$3$}
	(5.5,.4) node[right] {$2$};
	\draw (-.8,1.8) node[right] {$\bf{e_4}\longrightarrow\bf{e_3}\longrightarrow\bf{e_2}\longrightarrow\bf{e_1}\longrightarrow0$};
	\draw (-.8,1.1) node[right] {$\bf{e_7}\longrightarrow\bf{e_6}\longrightarrow\bf{e_5}\longrightarrow0$};
	\draw (-.8,.4) node[right] {$\bf{e_9}\longrightarrow\bf{e_8}\longrightarrow0$};
	\draw (-1.8,-.3) node[right] {$k_i\quad\;\;\quad 3\;\;\qquad 3\quad\;\;\quad 2\qquad\quad 1$};
	%\draw[rotate=0] (1.3,0) ellipse (100pt and 50pt);
	\end{tikzpicture}
\end{center}

Therefore, for $m(A)=(
4,3,2,0,0,0,0,0,0)^\top$ the dual is going to be $k(A)=(3,3,2,1,0,0,0,0,0)^\top$.
These numbers were introduced in \cite{GK78}, where the problem of complete description for the Jordan structure of a
matrix, which is a small perturbation of a given matrix, was posed. This problem was solved independently in \cite{DBT80} and \cite{MP80}.

\subsubsection{Backward Stability}

Although we are unable to obtain a general forward stability result, we can get the backward
stability result.

\begin{theorem} [Backward Stability]\label{main1}
	Let $A_0\in\mathbb{C}^{n\times n}$ be given. There
	exist constants $K,\epsilon>0$ $($depending on $A_0$ only$)$ such that for all $A$ with $\|A-A_0\|<\epsilon$ and for any factorization $U TU^*$ of $A$ $(U$ unitary and $T$ is upper triangular$)$ there exist $U_0$ and $T_0$ such that $A_0=U_0 T_0U_0^*$ is a Schur factorization of $A_0$ with
		\begin{equation}\label{eqback}
	\|U-U_0\|+\|T-T_0\|\leq K\|A-A_0\|^{1/n}.
	\end{equation}	
\end{theorem}

\subsubsection{Organization of the Paper}
In Section 2 we consider some auxillary results about what happens to GK numbers after we apply a reduction step. In Section 3 we discuss facts related to unitary Hessenberg matrices. In Section 4 we give a short overview about theory of gaps and semigaps. In Section 5 and 6 we present the proofs of Theorem \ref{main1} and Theorem \ref{GKnonstable} respectively.

\begin{center}
\definecolor{yqyqyq}{rgb}{0.7,0.7,0.7}
\definecolor{uququq}{rgb}{0.25098039215686274,0.25098039215686274,0.25098039215686274}
\begin{tikzpicture}[line cap=round,line join=round,>=triangle 45,x=0.9cm,y=0.9cm]
\clip(-8.2,-5.3) rectangle (7,7.3);
\draw [line width=2pt,color=uququq] (-3,3)-- (1,3);
\draw [line width=2pt,color=uququq] (1,3)-- (1,7);
\draw [line width=2pt,color=uququq] (1,7)-- (-3,7);
\draw [line width=2pt,color=uququq] (-3,7)-- (-3,3);
\draw [line width=2pt,color=uququq] (-3,-5)-- (1,-5);
\draw [line width=2pt,color=uququq] (1,-5)-- (1,-2.0033555448052507);
\draw [line width=2pt,color=uququq] (1,-2.0033555448052507)-- (-3,-2.0033555448052502);
\draw [line width=2pt,color=uququq] (-3,-2.0033555448052502)-- (-3,-5);
\draw [line width=2pt,color=yqyqyq] (-3,-0.5)-- (-3,1.5);
\draw [line width=2pt,color=yqyqyq] (-3,1.5)-- (1,1.5);
\draw [line width=2pt,color=yqyqyq] (1,1.5)-- (1,-0.5);
\draw [line width=2pt,color=yqyqyq] (1,-0.5)-- (-3,-0.5);
\draw [line width=2pt,color=yqyqyq] (2,1.5)-- (2,-0.5);
\draw [line width=2pt,color=yqyqyq] (2,-0.5)-- (6,-0.5);
\draw [line width=2pt,color=yqyqyq] (6,-0.5)-- (6,1.5);
\draw [line width=2pt,color=yqyqyq] (6,1.5)-- (2,1.5);
\draw [line width=2pt,color=yqyqyq] (2,6.5)-- (2,4);
\draw [line width=2pt,color=yqyqyq] (2,4)-- (6,4);
\draw [line width=2pt,color=yqyqyq] (6,4)-- (6,6.5);
\draw [line width=2pt,color=yqyqyq] (6,6.5)-- (2,6.5);
\draw [line width=2pt,color=yqyqyq] (2,-2.3)-- (2,-4);
\draw [line width=2pt,color=yqyqyq] (2,-4)-- (6,-4);
\draw [line width=2pt,color=yqyqyq] (6,-4)-- (6,-2.3);
\draw [line width=2pt,color=yqyqyq] (6,-2.3)-- (2,-2.3);
\draw [-{Stealth[length=4mm, width=3mm]},shift={(-5.4,1.1)},line width=1pt]  plot[domain=-0.37:0.37,variable=\t]({1*12.27*cos(\t r)+0*12.27*sin(\t r)},{0*12.27*cos(\t r)+1*12.27*sin(\t r)});
\draw [-{Stealth[length=4mm, width=3mm]},line width=1pt] (-1.059594184245494,1.5) -- (-1.059594184245494,3);
\draw [-{Stealth[length=4mm, width=3mm]},line width=1pt] (-1.0595941842454941,-0.5) -- (-1.0595941842454937,-2.0033555448052502);
\draw [-{Stealth[length=4mm, width=3mm]},line width=1pt] (2,0.5526191850498574) -- (1,0.5698647280818399);
\draw [-{Stealth[length=4mm, width=3mm]},line width=1pt] (2,5.3) -- (1,5);
\draw [-{Stealth[length=4mm, width=3mm]},line width=1pt] (2,-3.2) -- (1,-3.5);
\draw [line width=2pt,color=yqyqyq] (-4,-5)-- (-8,-5);
\draw [line width=2pt,color=yqyqyq] (-8,-5)-- (-8,-2.2);
\draw [line width=2pt,color=yqyqyq] (-8,-2.2)-- (-4,-2.2);
\draw [line width=2pt,color=yqyqyq] (-4,-2.2)-- (-4,-5);
\draw [line width=2pt,color=yqyqyq] (-4,-1)-- (-8,-1);
\draw [line width=2pt,color=yqyqyq] (-8,-1)-- (-8,1);
\draw [line width=2pt,color=yqyqyq] (-8,1)-- (-4,1);
\draw [line width=2pt,color=yqyqyq] (-4,1)-- (-4,-1);
\draw [line width=2pt,color=yqyqyq] (-8,6)-- (-8,4);
\draw [line width=2pt,color=yqyqyq] (-8,4)-- (-4,4);
\draw [line width=2pt,color=yqyqyq] (-4,4)-- (-4,6);
\draw [line width=2pt,color=yqyqyq] (-4,6)-- (-8,6);
\draw [-{Stealth[length=4mm, width=3mm]},line width=1pt] (-4,-3.5) -- (-3,-3.5);
\draw [-{Stealth[length=4mm, width=3mm]},line width=1pt] (-4,5) -- (-3,5);
\draw [-{Stealth[length=4mm, width=3mm]},line width=1pt] (-4,00) -- (-2,-2.0033555448052502);
\draw (-2.9,7) node[anchor=north west] {\parbox{3.4 cm}{\textit{    Section 6.    \\{\bf Proof of \\Theorem 1.6}}\\    Forward stability \\fails when \\GK numbers\\are different}};
\draw (-2.9,-2.2515721384004594) node[anchor=north west] {\parbox{3.4 cm}{\textit{    Section 5.\\{\bf    Proof of \\Theorem 1.7  } }\\ Backward stability }};
\draw (-3,1.5) node[anchor=north west] {\parbox{3.4 cm}{   \small   {\bf Proposition 4.4  }\\  Semi-gap between invariant     subspaces of $A$ and $A_0$}};
\draw (-8.05,6) node[anchor=north west] {\parbox{3.4 cm}{ \small   {\bf   Lemma 2.2 }\\   GK numbers change     in the reduction step}};
\draw (-8.05,1) node[anchor=north west] {\parbox{3.4 cm}{ \small   {\bf   Lemma 3.3  }\\  Existence of Schur     form using     Hessenberg matrices}};
\draw (-8.05,-2.25) node[anchor=north west] {\parbox{3.4 cm}{ \small   {\bf Lemma 3.4}    \\Representation of unitary     matrices as a product of unitary Hessenberg matrices}};
\draw (1.95,6.5) node[anchor=north west] {\parbox{3.4 cm}{  \small   {\bf Lemma 6.2  }\\  Particular case of Theorem 1.6    when (1,1)-entry $ \lambda$ of $T_0$    is with $\ker(A_0- \lambda I) \ge2$}};
\draw (2,-2.3) node[anchor=north west] {\parbox{3.4 cm}{   \small   {\bf Proposition 1.2  }\\  Stability of \\eigenvalues}};
\draw (1.95,1.5) node[anchor=north west] {\parbox{3.4 cm}{   \small   {\bf Lemma 4.3  \\  Lemma 4.5  \\  Lemma 4.6   }\\ Properties of semigaps}};
\end{tikzpicture}
\end{center}

\section{Auxiliary Results}

Before proving Theorem \ref{GKnonstable} we need a couple of technical lemmas. 

Let us start by introducing the following fact.
\begin{lemma}\label{basisfromvector}
    For every eigenvector $x$ of $A$ there is a Jordan basis of $A$ including $x$.
\end{lemma}
\begin{proof} 
Let us fix any Jordan basis of $A$, with the Jordan chains  corresponding to $\lambda_t$ ordered by length. Now, given another eigenvector $x$ decompose it in that basis. Look for the last non-zero coefficient, say, $\alpha$ that is corresponding to the eigenvector, say, $y$. 

Then the chain for $x$ has the same length as for $y$ and we can replace the chain for $y$ with the chain for $x$. All it remains to prove is the linear independence of the new set of vectors.

Let $Y$ stands for the matrix whose columns are the Jordan basis we started with. $Y$ is invertible. Denote by $X$ the matrix where the chain for $y$ is replaced by the chain for $x$. 
Then $X=YR$, where $R$ is an upper triangular matrix that is invertible, since it has either 1 on its diagonal or $\alpha$.
Note that for the generalized eigenvectors of the chain for $x$ we have the same decompositions with the same coefficients as for $x$ with the vectors from the corresponding chains for the original basis, so we can write down the matrix $R$.
\end{proof}
The next result describes the recursion we will use. In particular, we want to figure out what happens to the GK numbers during each step of recursion.
Here is the idea behind it:
\begin{center}\begin{tikzpicture}[line cap=round,line join=round,>=triangle 45,x=1cm,y=1cm]
\clip(-12.12,2.5) rectangle (0,5);
\draw (-12.105249153459209,4.8848968485047015) node[anchor=north west] {$m_1(A,\lambda_t)\ge m_2(A,\lambda_t)\ge m_3(A,\lambda_t)\ge \ldots \ge m_{l-1}(A,\lambda_t)\ge m_l(A,\lambda_t)$};
\draw [line width=1pt] (-12,4.2)-- (-12,4);
\draw [line width=1pt] (-12,4)-- (-2.8,4);
\draw [line width=1pt] (-2.8,4)-- (-2.8,4.2);
%%%%%%
\draw (-12.0921026287238,3.35) node[anchor=north west] {$m_j(A,\lambda_t)<m_l(A,\lambda_t)\text{ for }j>l$};
\draw [shift={(-11,2)},line width=1pt]  plot[domain=1.5882848543070331:2.1641892860463714,variable=\t]({1*1.5573359986717992*cos(\t r)+0*1.5573359986717992*sin(\t r)+0.1},{0*1.5573359986717992*cos(\t r)+1*1.5573359986717992*sin(\t r)-0.057});
\draw [shift={(-10.369907888385308,5.345025214243993)},line width=1pt]  plot[domain=4.073352330911641:4.695441450578285,variable=\t]({1*1.6752367373558714*cos(\t r)+0*1.6752367373558714*sin(\t r)},{0*1.6752367373558714*cos(\t r)+1*1.6752367373558714*sin(\t r)});
\draw [line width=1pt] (-10.92,3.5)-- (-10.396200937856124,3.5);
\draw [-{Stealth[length=4mm, width=2mm]},line width=1pt] (-10.396200937856124,3.67) -- (-10,3.67);
\draw [-{Stealth[length=4mm, width=2mm]},line width=1pt] (-10.422645098966939,3.5) -- (-10,3.5);
\draw (-9.988658671058467,3.9) node[anchor=north west] {$\text{The corresponding Jordan chains stay the same}.$};
\end{tikzpicture}
\end{center}

So what happens when $m_j(A,\lambda_t)=m_l(A,\lambda_t)$ for some $j$'s greater than $l$? Let $j^*$ be the maximal such index.

\definecolor{cqcqcq}{rgb}{0.7529411764705882,0.7529411764705882,0.7529411764705882}
\begin{tikzpicture}[line cap=round,line join=round,>=triangle 45,x=1cm,y=1cm]
\clip(-14.21201668692858,-1.3) rectangle (1.0090665963473844,1.4648284385882258);
\fill[line width=1pt,color=cqcqcq,fill=cqcqcq,fill opacity=0.05] (-13,-0.043) -- (-13,-0.65) -- (-11.43,-0.65) -- (-11.43,-0.043) -- cycle;
\draw (-13.090581182378063,0) node[anchor=north west] {$m_l(A,\lambda_t)=m_{l+1}(A,\lambda_t)=\ldots=m_{j^*}(A,\lambda_t)>m_l(A,\lambda_t)-1$};
\draw [line width=1pt,color=cqcqcq] (-13,-0.04338353162684747)-- (-13.005786097834683,-0.6589926324169902);
\draw [line width=1pt,color=cqcqcq] (-13.005786097834683,-0.6589926324169902)-- (-11.40848334248257,-0.6392728453138777);
\draw [line width=1pt,color=cqcqcq] (-11.40848334248257,-0.6392728453138777)-- (-11.40848334248257,-0.04570725351019176);
\draw [line width=1pt,color=cqcqcq] (-11.40848334248257,-0.04570725351019176)-- (-13,-0.04338353162684747);
\draw [shift={(-10,-3)},line width=1pt]  plot[domain=1.567619094245166:2.1657704681284176,variable=\t]({1*3.5695681211418098*cos(\t r)+0*3.5695681211418098*sin(\t r)},{0*3.5695681211418098*cos(\t r)+1*3.5695681211418098*sin(\t r)});
\draw [line width=1pt] (-9.98865867105847,0.5695501041069182)-- (-5.238161957918663,0.5754660402378526);
\draw [{Stealth[length=4mm, width=2mm]}-,shift={(-5.780456103254257,-3.2856682745515724)},line width=1pt]  plot[domain=1.007480065302928:1.4312596139915073,variable=\t]({1*3.8996944747832263*cos(\t r)+0*3.8996944747832263*sin(\t r)},{0*3.8996944747832263*cos(\t r)+1*3.8996944747832263*sin(\t r)});
\draw [rotate around={-0.4:(-3,-0.34643400683265624)},line width=1pt,color=cqcqcq,fill=cqcqcq,fill opacity=0.05] (-3.95,-0.37) ellipse (1.35cm and 0.35cm);
\draw (-11.875842296826331,1.2) node[anchor=north west] {$\text{Recursion decreases this chain by one vector.}$};
\end{tikzpicture}
{%\color{blue}

Now let us formalize it.

\begin{lemma}\label{reducingJordan}
Consider matrix $B$ with the eigenvalues $\{\lambda_j\}$'s, having the GK numbers $\{m_j(B,\lambda_i)\}$ and $e_1$as its eigenvector corresponding to the Jordan chain for $\lambda_t$ and $m_l(B,\lambda_t)$, i.e. 
	\begin{equation*}
B=\left[\begin{array}{c|c}
\lambda_t&\begin{matrix}
\star&\cdots&\star
\end{matrix}\\ \hline\\
\begin{matrix}
0\\ \vdots \\0
\end{matrix}&\text{\huge $C$}
\end{array}\right].
\end{equation*} 
Then 
\begin{itemize}
    \item $m_j(C,\lambda_i)=m_j(B,\lambda_i)$ for all $i\ne t$ or $i=t$ and $j>l+1$;
    \item $m_l(C,\lambda_t)=m_l(B,\lambda_t)-1$, $m_{l+1}(C,\lambda_t)=m_{l+1}(B,\lambda_t)$ if $m_l(A,\lambda_t)>m_{l+1}(B,\lambda_t)$;
    \item  $m_{j^*}(C,\lambda_t)=m_l(B,\lambda_t)-1$, $m_{j}(C,\lambda_t)=m_{j+1}(B,\lambda_t)$ for $j=l,\dots,j^*-1$ if $m_l(B,\lambda_t)=m_{l+1}(B,\lambda_t)=\ldots=m_{j^*}(B,\lambda_t)$ and $j^*$ is the maximal such index;
    \item $m_j(C,\lambda_t)=m_j(B,\lambda_t)$ if $m_j(B,\lambda_1)<m_{l}(B,\lambda_t)$.
\end{itemize}
\end{lemma}
\begin{proof}
Note that due to Lemma \ref{basisfromvector} there is a Jordan basis of $B$ containing $e_1$.	Let $J$ be the canonical Jordan form of $B$ where the first block corresponds to the Jordan chain for $\lambda_t$ that we mentioned. Thus, there is a invertible matrix $R$ containing the Jordan basis $\{f_{i,j}^{(k)}\}_{i,j,k}$ ($i$ is the place in the Jordan chain for  $f_{i,j}^{(k)}$ corresponding to $m_j(B,\lambda_k)$) as its columns, where the first $m_l(B,\lambda_t)$ vectors forms the chain  of $A$, having $f_{0,l}^{(t)}=e_1$, i.e. $R=\left [f_{0,l}^{(t)}|f_{1,l}^{(t)}|\ldots|f_{m_l(B,\lambda_t)-1,l}^{(t)}|\ldots\right]$.
%	$$A=PJP^{-1}\qquad\text{ or } \qquad B=U^*PJP^{-1}U=U^*PJ(U^*P)^{-1}.$$
%%	Hence, columns of $R=U^*P$ form the Jordan chains for $B$ corresponding to the same eigenvalues and of the same length. Moreover, the first column of $R$ is $e_1$. 
That is $R$ and $R^{-1}$ are of the following form.
	$$R=\left[\begin{array}{c|c}
1&\begin{matrix}
	\blacklozenge&\cdots&\blacklozenge
	\end{matrix}\\ \hline\\
	\begin{matrix}
	0\\ \vdots \\0
	\end{matrix}&\text{\huge $R_1$}
	\end{array}\right]\qquad\text{ and }\qquad R^{-1}=\left[\begin{array}{c|c}
1&\begin{matrix}
	\clubsuit&\cdots&\clubsuit
	\end{matrix}\\ \hline\\
	\begin{matrix}
	0\\ \vdots \\0
	\end{matrix}&\text{\huge $R_1^{-1}$}
	\end{array}\right].$$
This argument implies that 
\begin{equation}\left[\begin{array}{c|c}\label{chainpreserve}
\lambda_t&\begin{matrix}
1&0&\cdots&0
\end{matrix}\\ \hline\\
\begin{matrix}
0\\ \vdots \\0
\end{matrix}&\text{\huge $J_1$}
\end{array}\right]=J=R^{-1}BR=\end{equation}
$$=\left[\begin{array}{c|c}
1&\begin{matrix}
	\blacklozenge&\cdots&\blacklozenge
	\end{matrix}\\ \hline\\
	\begin{matrix}
	0\\ \vdots \\0
	\end{matrix}&\text{\huge $R_1$}
	\end{array}\right]\left[\begin{array}{c|c}
\lambda_t&\begin{matrix}
\star&\cdots&\star
\end{matrix}\\ \hline\\
\begin{matrix}
0\\ \vdots \\0
\end{matrix}&\text{\huge $C$}
\end{array}\right]\left[\begin{array}{c|c}
1&\begin{matrix}
	\clubsuit&\cdots&\clubsuit
	\end{matrix}\\ \hline\\
	\begin{matrix}
	0\\ \vdots \\0
	\end{matrix}&\text{\huge $R_1^{-1}$}
	\end{array}\right].$$
	%Therefore, if $(B-\lambda_iI)x=y$ then
	%\begin{equation}\label{chainpreserve}
	%	(B-\lambda_iI)x=\left[\begin{array}{c|c}
	%\lambda_t-\lambda_i&\begin{matrix}
%	\star&\cdots&\star
%	\end{matrix}\\ \hline\\
%	\begin{matrix}
%	0\\ \vdots \\0
%	\end{matrix}&\text{\huge $C$}
%	\end{array}\right]\begin{bmatrix}
%	x_1\\x_2\\\vdots\\x_n
%	\end{bmatrix}=\left[\begin{array}{c}
%\star
%\\ \hline\\
%\text{\huge $C$}\begin{bmatrix}
%x_2\\\vdots\\x_n
%\end{bmatrix}
%	\end{array}\right]=\left[\begin{array}{c}
%	y_1
%	\\ \hline\\
%	y_2\\\vdots\\y_n
%	\end{array}\right].
%	\end{equation}
Note that $J_1=R_1^{-1}CR_1$ is the Jordan form of $C$.

So what is the difference between $J$ and $J_1$? The only Jordan chain that is affected is
$$0\leftarrow f_{0,l}^{(t)}\leftarrow f_{1,l}^{(t)}\leftarrow \ldots\leftarrow f_{m_l(B,\lambda_t)-1,l}^{(t)}.$$
We delete the eigenvector from this chain and truncate the rest of the vectors to get a Jordan chain of length $m_l(B,\lambda_t)-1$ of $C$. The length of the rest Jordan chains of $C$ stay the same as they were in $B$.
The conclusion of the lemma follows from this observation. 
\end{proof}

\section{Unitary Hessenberg Matrices and Schur Canonical Forms}

Throughout  this paper we are going to use the special type of structured matrices that are called Hessenberg. So, let us introduce it to the reader first.

	A matrix is called the \textit{upper Hessenberg} if it has zero entries below the first subdiagonal.
	Similarly, it is called the \textit{lower Hessenberg} if it has zeros above the first super diagonal.

%Hessenberg matrices that are also unitary have a number of interesting characteristics like the QR algorithm (see~\cite{G86}), and are
%of particular interest for us during the construction of the Schur form.

The following is a well-known fact (for example see~\cite{S05}).

%Let us discuss a physical aspect of this topic  (see~\cite{O98}). A simple physical device known
%to electrical engineers as a \textit{"discrete transmission line"}, and to geophysicists as a \textit{"layered earth model"}, can be useful to study the algebraic properties of
%Szeg\H{o} polynomials $\Phi^\sharp=\{\phi^\sharp_k(x)\}$, i.e. polynomials orthogonal with respect to a suitable inner product on
%the unit circle,
%$$\langle p(x),q(x)\rangle=\frac{1}{2\pi}\int_{-\pi}^{\pi}p(e^{i\theta})\cdot \overline{\left[q(e^{i\theta})\right]}w^2(\theta)d\theta.$$
%We use the sharp sign $\sharp$ to follow the usual signal processing designations, where $\left\{\overline{\left[\phi^\sharp_k(\frac{1}{\bar{z}})\right]}\right\}$
%are called
%\textit{backward predictor polynomials} (see \cite{MG76}). The key point is that the first $(n+1)$ Szeg\H{o} polynomials $\Phi^\sharp$ can be expressed by only $(n+1)$ parameters $\{\mu_0,\rho_1,\dots,\rho_n \}$ via the two-term recurrence relations (see~\cite{G48})
%\begin{equation}\label{recrel}
%\begin{bmatrix}
%\phi_0(z)\\ \phi_0^\sharp(z)
%\end{bmatrix}=\frac{1}{\mu_0}\begin{bmatrix}
%1\\1
%\end{bmatrix}\qquad
%\begin{bmatrix}
%\phi_{k+1}(z)\\ \phi_{k+1}^\sharp(z)
%\end{bmatrix}=\frac{1}{\mu_{k+1}}\begin{bmatrix}
%1&-\bar{\rho}_{k+1}\\-{\rho_{k+1}}&1
%\end{bmatrix}\begin{bmatrix}
%1&0\\0&z
%\end{bmatrix}
%\begin{bmatrix}
%\phi_{k}(z)\\ \phi_{k}^\sharp(z)
%\end{bmatrix}.
%\end{equation}
%
%The auxiliary polynomials  involved in these recurrence relations have a reversal form $\phi_{k}(z)=z^k\overline{\left[\phi^\sharp_k(\frac{1}{\bar{z}})\right]}$.
%

\begin{proposition}\label{unitaryhessenberg}
	An $n\times n$ lower unitary Hessenberg matrix can be represented in the following way
	\begin{equation*}\label{GGT}
	U=\begin{bmatrix}
	-\rho_1&\mu_1&0&\dots&0\\
	-\rho_2\mu_1&-\rho_2\bar{\rho}_1&\mu_2&\dots&0\\
	\vdots&\vdots&\vdots&\ddots&\vdots\\
	\vdots&\vdots&\vdots&\ddots&0\\
	\vdots&\vdots&\vdots&\ddots&\mu_{n-1}\\
	-\rho_n\mu_{n-1}\ldots\mu_1&-\rho_n\mu_{n-1}\ldots\mu_2\bar{\rho}_1&-\rho_n\mu_{n-1}\ldots\mu_3\bar{\rho}_2&\dots&-\rho_n\bar{\rho}_{n-1}
	\end{bmatrix},
	\end{equation*}
	where $\mu_j=\sqrt{1-\rho_j^2}$ for all $j$'s.
%	where $\rho_i\in\mathbb{C}^n$ and $\mu_i=\sqrt{1-|\rho_i|^2}\in\mathbb{R}^+$.
\end{proposition}

The following statement is an immediate consequence of the previous proposition.

\begin{corollary}
	If  $\rho_j<1$ for all $j$, the first column of a lower unitary Hessenberg matrix completely defines the whole matrix.
\end{corollary}

Let us consider the following properties of Hessenberg matrices first. Unitary Hessenberg matrices have a number of interesting properties, and are of particular
importance in their relationship with the Schur form. In the classical proof of the construction of the Schur form, one
builds an orthonormal set using eigenvectors of the matrix $A_0$, typically through the Gram-Schmidt process. With the above observation, we can derive the Schur form specifically through
the use of unitary Hessenberg matrices.

\begin{lemma}\label{schurhess}
	For any $A\in\mathbb{C}^{n\times n}$ there exist $U$ unitary and $T$ upper triangular with the eigenvalues of $A$ along the diagonal such that $A=UTU^*$. 
\end{lemma}

\begin{proof}
	Let $\lambda_1,\dots,\lambda_m$ be the eigenvalues of $A$ and let $x$ be a unit eigenvector of $A$ corresponding to eigenvalue $\lambda_1$. Moreover, pick $H_1$ to be a lower unitary Hessenberg matrix with $x$ as its first column. By Proposition \ref{unitaryhessenberg} this determines $H_1$ completely.
	
	Then, we have $H_1^*AH_1e_1=H_1^*Ax=H_1^*\lambda_1x=\lambda_1e_1$. In other words,
	\begin{equation}
	H_1^*AH_1=\left[\begin{array}{c|c}
	\lambda_1&\begin{matrix}
	\star&\cdots&\star
	\end{matrix}\\ \hline\\
	\begin{matrix}
	0\\ \vdots \\0
	\end{matrix}&\text{\huge $A_2$}
	\end{array}\right].
	\end{equation}
	
	By repeating the process of reducing the matrix dimensions, i.e. for each matrix $A_k$ constructing matrix $H_k$ in a way we described, we get a string of matrices $H_1,\dots,H_n$ that are all unitary Hessenberg and
	\begin{equation}\label{schhesseq}
	\widetilde{H}_{n-1}^*\dots\widetilde{H}_1^*A_0 \widetilde{H}_1\dots \widetilde{H}_{n-1}=T_0,
	\end{equation}
	Where $T_0$ is an upper triangular matrix, $\widetilde{H}_1=H_1$, and $\widetilde{H}_k=\left[\begin{smallmatrix}
	I_{k-1}&\textbf{0}\\
	\textbf{0}&H_k
	\end{smallmatrix}\right]$ for $k=2,\dots,n-1$ with $I_j$ being the $j\times j$ identity matrix. By taking $U=\widetilde{H}_1\cdot\ldots\cdot \widetilde{H}_{n-1}$ we get the result.
\end{proof}

Observe that we constructed our unitary matrix $U$ using only unitary Hessenberg matrices $H_k$,
each of whose first column was an eigenvector of the corresponding matrix $A_k$.

\begin{lemma}\label{hessfactor}
	Every unitary matrix $U$ admits a factorization $$U=H_1\cdot\ldots\cdot H_{n-1},$$ where matrices $H_i$'s are unitary Hessenberg.
\end{lemma}
\begin{proof}
	Let $H_1$ be the unitary Hessenberg matrix whose first column $x_1$ is the same as $U$. Note that $x_1^*x_j=\delta_{1,j}$, since $U$ is unitary. Then
	\begin{equation}
	H_1^*U=\left[\begin{array}{c|c}
	1&\begin{matrix}
	0&\cdots&0
	\end{matrix}\\ \hline\\
	\begin{matrix}
	0\\ \vdots \\0
	\end{matrix}&\text{\huge $U_2$}
	\end{array}\right],
	\end{equation}
	where $U_2$ is unitary. As before we repeat the process until we get $$H^*_{n-1}\cdot\ldots\cdot H^*_{1}U=I.$$  The result follows from simply multiplying the both parts of this equality by $H_1\cdot\ldots\cdot H_{n-1}$. 
\end{proof}

%%%%%%%%%%%%%%%%%%%%%%%%%%%%%%

\section{Gap and Semi-gap}

In this section we discuss some topological properties of the set of subspaces in $\mathbb{C}^n$, since in order to prove our main result, we require some facts from the theory of gaps. We begin by
stating some definitions.
%\begin{definition}

	A matrix $P_\mathcal{M}$ is called an\textit{ orthogonal projector} onto a subspace $\mathcal{M}\subset \mathbb{C}^n$ if 
	\begin{itemize}
		\item $\mathrm{Im}P_\mathcal{M}=\mathcal{M}$;
		\item $P_\mathcal{M}^2=P_\mathcal{M}$;
		\item $P_\mathcal{M}^*=P_\mathcal{M}$.
	\end{itemize}
%\end{definition}

The following concept is the key definition.

%\begin{definition}
	Let $\mathcal{M},\mathcal{N}$ be subspaces of $\mathbb{C}^n$, and let $P_\mathcal{M},P_\mathcal{N}$ be the orthogonal projectors onto $\mathcal{M}$ and $\mathcal{N}$
	respectively. We define  \textit{the gap} $\theta(\mathcal{M},\mathcal{N})$ between $\mathcal{M}$ and $\mathcal{N}$ as follows
	$$\theta(\mathcal{M},\mathcal{N})=\|P_\mathcal{M}-P_\mathcal{N}\|$$
	or, equivalently,
	$$\theta(\mathcal{M},\mathcal{N})=\max\left\{\underset{\begin{smallmatrix}
		x\in\mathcal{M}\\ \|x\|=1
		\end{smallmatrix}}{\sup}\underset{y\in\mathcal{N}}{\inf}\|x-y\|,\underset{\begin{smallmatrix}
		y\in\mathcal{N}\\ \|y\|=1
		\end{smallmatrix}}{\sup}\underset{x\in\mathcal{M}}{\inf}\|x-y\|\right\}.$$
%\end{definition}

It follows immediately from the definition that $\theta(\mathcal{M},\mathcal{N})$ is a metric on the set of all subspaces in $\mathbb{C}^n$. Moreover, $\theta(\mathcal{M},\mathcal{N})\leq 1$.

Note that the Hausdorff distance between sets $\text{Inv } A$ and  $\text{Inv }  B$ of
all invariant subspaces matrices $A$ and $B$ can be defined as follows
$$\text{dist }(\text{Inv }A,\text{Inv }B)=\max\{\underset{\mathcal{M}\in\text{Inv }A}{\sup}\theta(\mathcal{M},\text{Inv }B),\underset{\mathcal{N}\in\text{Inv }B}{\sup}\theta(\mathcal{N},\text{Inv }A)\}.
$$ This distance is a metric as well.

We are going to use the following property of gaps between subspaces. It can be found in \cite{GLR86}.
\begin{proposition}\label{orthogap}
	For subspaces  $\mathcal{M},\mathcal{N}\subset\mathbb{C}^n$, we have
	\begin{equation}\label{gap}
	\theta(\mathcal{M},\mathcal{N})=\theta(\mathcal{N}^\perp,\mathcal{M}^\perp).
	\end{equation}
\end{proposition}

The symmetry with respect to subspaces of the gap is actually
a disadvantage.

\begin{proposition}
	Let $\mathcal{M},\mathcal{N}$ be subspaces of $\mathbb{C}^n$.
	\begin{itemize}
		\item [(i)]If $\dim(\mathcal{M})=\dim(\mathcal{N})$ then for any $x\in\mathcal{M}$ there exists a $y\in\mathcal{N}$ such that $\|x-y\|\leq \theta(\mathcal{M},\mathcal{N})$.
		\item[(ii)] If $\dim(\mathcal{M})\neq\dim(\mathcal{N})$ then $ \theta(\mathcal{M},\mathcal{N})=1$.
	\end{itemize}
\end{proposition}

The above result shows us that the gap is often not useful to consider when  $\dim(\mathcal{M})\neq\dim(\mathcal{N})$.
In our theorem, we wish to find bounds on the kernels of the matrices, however, the dimension of the kernels are, in general, not equal. 

The gap provides many useful results in providing a variety of bounds but the
usefulness is limited to when the dimensions are equal. The concept of a semi-gap can be helpful when the dimensions are not equal. This advantage is highly useful when considering matrix
perturbations.

%\begin{definition} 
Let $\mathcal{M},\mathcal{N}$ be subspaces of $\mathbb{C}^n$. The quantity
	$$\theta_0(\mathcal{M},\mathcal{N})=\underset{\begin{smallmatrix}
		x\in\mathcal{M}\\ \|x\|=1
		\end{smallmatrix}}{\sup}\underset{y\in\mathcal{N}}{\inf}\|x-y\|$$
	is called the \textit{semigap} (or one-sided gap) from $\mathcal{M}$ to $\mathcal{N}$.
%\end{definition}

We notice some immediate properties of the semi-gap.

\begin{lemma}\label{semigap}
	Let $\mathcal{M},\mathcal{N}\subset \mathbb{C}^n$ be two subspaces. Then the following statements hold.
	\begin{itemize}
		\item [(i)] $\theta(\mathcal{M},\mathcal{N})=\max\{\theta_0(\mathcal{M},\mathcal{N}),\theta_0(\mathcal{N},\mathcal{M})\}$.
		\item[(ii)] $\theta_0(\mathcal{M},\mathcal{N})=\underset{\begin{smallmatrix}
			x\in\mathcal{M}\\ \|x\|=1
			\end{smallmatrix}}{\sup}{}\|x-P_\mathcal{N}x\|$.
		\item [(iii)]If $\mathcal{N}_1\subset\mathcal{N}_2$, then $\theta_0(\mathcal{M},\mathcal{N}_2)\leq \theta_0(\mathcal{M},\mathcal{N}_1)$,  $\theta_0(\mathcal{N}_1,\mathcal{M})\leq \theta_0(\mathcal{N}_2,\mathcal{M})$.
		\item[(iv)] $\theta_0(\mathcal{M},\mathcal{N})\leq 1$.
		\item[(v)] If $\dim \mathcal{M}>\dim\mathcal{N}$, then $\theta_0(\mathcal{M},\mathcal{N})=1$.
		\item[(vi)] $\theta_0(\mathcal{M},\mathcal{N})<1$ if and only if $\mathcal{M}\cap\mathcal{N}^\perp=\emptyset$.
	\end{itemize} 
\end{lemma}

These facts are well-known and can be found e.g. in \cite{GLR86,K66}.

 To this end we will
need some new results on gap and semigap. These results will be derived next.

\begin{proposition}\label{kernel}
	Let $A_0$ be fixed. Then, there exist $\epsilon, K>0$ such that for all $A$ with $\|A-A_0\|<\epsilon$, we have
	\begin{equation}\label{ker}
	\theta_0(\ker (A),\ker (A_0))\leq K\|A-A_0\|.
	\end{equation}
\end{proposition}

Since $\ker(A)\oplus\mathrm{Im}(A^\top ) = \mathbb{C}^n$, often times it is easier to prove a result for the image rather than for
the kernel. Because of this, the above proposition will follow from the next results.

\begin{lemma}\label{image}
	Let $A_0$ be fixed. Then, there exist $\epsilon, K>0$ such that for all $A$ with $\|A-A_0\|<\epsilon$, we have
	\begin{equation*}
	\theta_0(\mathrm{Im} (A_0),\mathrm{Im} (A))\leq K\|A-A_0\|.
	\end{equation*}
\end{lemma}

\begin{proof}
	Let $A_0$ be an $n\times n$-matrix with $\dim(\mathrm{Im}(A_0))=k$. Consider an orthonormal basis $g_1,g_2,\dots,g_k$ of $\mathrm{Im}(A_0)$. That is there are $f_1,f_2,\dots,f_k$ such that $g_i=A_0f_i$ for $i=1,\dots,k$. Define $h_i=Af_i$, so that $h_i\in\mathrm{Im}(A)$ for all $i$. Then, we have that
	$$\|h_i-g_i\|=\|Af_i-A_0f_i\|=\|(A-A_0)f_i\|\leq\|f_i\| \|A-A_0\|.$$
	Now, let $x\in\mathrm{Im}(A_0)$ and $\|x\|=1$. This means that $x=\alpha_1g_1+\dots\alpha_kg_k$ for some $\alpha_i$'s. Define $y=\alpha_1h_1+\dots\alpha_kh_k$. Clearly, $y\in\mathrm{Im}(A)$ and
	\begin{align*}
	\inf_{z\in\mathrm{Im}A}\|x-z\|\leq\|x-y\|&\leq\sum_{i=1}^{k}\alpha_i\|h_i-g_i\|\\
	&\leq\underset{j}{\max}(|\alpha_j|)\sum_{i=1}^k\|f_i\| \|A-A_0\|\\
	&\leq k\underset{j}{\max}(|\alpha_j|)\underset{i}{\max}\|f_i\| \|A-A_0\|.
	\end{align*}
	Therefore, by taking the supremum over $x$ we arrive at $	\theta_0(\mathrm{Im} (A_0),\mathrm{Im} (A))\leq  K\|A-A_0\|$, where $K=k\underset{j}{\max}(|\alpha_j|)\underset{i}{\max}\|f_i\| $.
\end{proof}

Next, in order to show that Lemma~\ref{image} implies Proposition~\ref{kernel}, we need the following result.

\begin{lemma}\label{ortho}
	Let  $\mathcal{M},\mathcal{N}$ be subspaces of $\mathbb{C}^n$. Then we have
	\begin{equation}\label{gap}
	\theta_0(\mathcal{M},\mathcal{N})=\theta_0(\mathcal{N}^\perp,\mathcal{M}^\perp).
	\end{equation}
\end{lemma}
\begin{proof}
	First, notice that if $\dim\mathcal{M}> \dim \mathcal{N}$ or $\theta_0(M,N) =1$, the result follows immediately.
	Now, we consider the case where $\dim\mathcal{M}\leq \dim \mathcal{N}$  and $\theta_0(M,N) < 1$. Define $P$ to be the subspace of all the projection of vectors in $\mathcal{M}$ to $\mathcal{N}$, i.e. $P=\text{proj}_\mathcal{N} \mathcal{M}$.  Since  $\theta_0(M,N) < 1$, we have that $\dim P=\dim\mathcal{M}$. Additionally, we have that $\theta(M,P) =\theta_0(M,N) $. Recall that $P\subset \mathcal{N}$ implies that $\mathcal{N}^\perp\subset P^\perp$. By Proposition~\ref{orthogap} we get $\theta(\mathcal{M},P) =\theta(P^\perp,\mathcal{M}^\perp) $. Then, using Lemma~\ref{semigap}, we have $\theta(P^\perp,\mathcal{M}^\perp) \geq\theta_0(P^\perp,\mathcal{M}^\perp) =\theta(\mathcal{N}^\perp,\mathcal{M}^\perp) $. That is, 
	$$\theta(\mathcal{N}^\perp,\mathcal{M}^\perp)\leq \theta(\mathcal{M},\mathcal{N}).
	$$ Now, repeating the argument for $\mathcal{N}^\perp$ and $\mathcal{M}^\perp$ gives us the result.
\end{proof}

Recall that $\ker(A)=\mathrm{Im}(A^*)^\perp$ and $\mathrm{Im}(A)=\ker(A^*)^\perp$. So, by combining Lemma~\ref{image} and Lemma~\ref{ortho}, we see that Proposition~\ref{kernel} holds true.

%%%%%%%%%%%%%%%%%%%%%%%%%%%%%%%%%%%%%%%%%%%%%%%%%%%%%%%%%%%%%%%%%%

\section{Backward Stability}

In this section we finalize the proof of Theorem~\ref{main1}.

\begin{proof}[Proof of Theorem~\ref{main1}]
	Let $A_0$ be given. Then, according to Proposition~\ref{kernel} there exists $K,\epsilon> 0$ such that \eqref{ker} holds.
	Let $A$ be a matrix such that $\|A-A_0\|<\epsilon$ and $A$ has Schur decomposition $A = UTU^*$.
	Let $\lambda_i$'s denote the eigenvalues of $A_0$ and $\mu_j$'s denote the eigenvalues of $A$.
	In the Schur decomposition for $A$, we have a sequence of unitary Hessenberg matrices $\{V_k\}_1^n$ such that $U = V_1\cdot\ldots\cdot V_n$.
	Since the Schur decomposition construction is an iteration of stepsas was shown in Lemma~\ref{schurhess}, it is enough to show that we can obtain a bound
	on the first step.
	Let $\mu_1$ be an eigenvalue of $A$ and let $V_1$ be the unitary matrix such that
	\begin{equation*}
	V_1^*AV_1=\left[\begin{array}{c|c}
	\mu_1&\begin{matrix}
	\star&\cdots&\star
	\end{matrix}\\ \hline\\
	\begin{matrix}
	0\\ \vdots \\0
	\end{matrix}&\text{\huge $A_2$}
	\end{array}\right]
	\end{equation*} 
	where the first column $v_1$ of $V_1$ is a unit eigenvector of $A$ corresponding to eigenvalue $\mu_1$, and the matrix forms an orthogonal basis for $\mathbb{C}^n$.
	
	Thus, by Proposition~\ref{kernel} we can find a unit vector $u_1\in\ker(A_0-\lambda_1I)$ and constants $K_i$'s such that $\|v_1-u_1\|\leq K_0\|A-A_0+(\lambda_1-\mu_1)I\|\leq K_0\|A-A_0\|+K_1|\lambda_1-\mu_1|\leq K_0\|A-A_0\|+\widetilde{K}_1\|A-A_0\|^{\alpha}\leq K_2\|A-A_0\|^\alpha$ where $\alpha$ is either 1 or ${1/n}$.
	
	Hence, we can find  a corresponding orthonormal basis forming $U_1$ such that  $\|v_i-u_i\|\leq K_i\|A-A_0\|^\alpha$.
	Therefore, we have that $\|V_1-U_1\|\leq M_1\|A-A_0\|^\alpha$ and by construction, we have that
	\begin{equation*}
	U_1^*A_0U_1=\left[\begin{array}{c|c}
	\lambda_1&\begin{matrix}
	\star&\cdots&\star
	\end{matrix}\\ \hline\\
	\begin{matrix}
	0\\ \vdots \\0
	\end{matrix}&\text{\huge $A_{0,2}$}
	\end{array}\right]
	\end{equation*}
	
	Repeating the process the same way we did in the proof of Lemma~\ref{hessfactor}, we acquire a unitary matrix $U_0=U_1\cdot\ldots\cdot U_{n}$ such that  $\|V_i-U_i\|\leq M_i\|A_i-A_{0,i}\|^\alpha$ for all $i=1,\dots,n$.
	
	Note that by construction
	$$\|A_{i+1}-A_{0,i+1}\|\leq\|V_i^*A_iV_i-U_i^*A_{0,i}U_i\|\le\|V_i^*A_iV_i-V_i^*A_{0,i}V_i+V_i^*A_{0,i}V_i-$$
	$$-V_i^*A_{0,i}U_{i}+V_i^*A_{0,i}U_{i}-U_i^*A_{0,i}U_i\|\leq\|V_i^*A_iV_i-V_i^*A_{0,i}V_i\|+\|V_i^*A_{0,i}V_i-$$
	$$-V_i^*A_{0,i}U_{0,i}\|+\|V_i^*A_{0,i}U_{0,i}-U_i^*A_{0,i}U_i\|\leq\|A_i-A_{0,i}\|+\|A_{0,i}\|\|V_i-$$
	$$-U_{0,i}\|+\|V_i^*-U_i^*\|\|A_{0,i}\|\leq \|A_i-A_{0,i}\|+2M_i\|A_{0,i}\|\|A_i-A_{0,i}\|^\alpha\leq \widetilde{M}_i\|A-A_0\|^\alpha.$$
	It follows by induction that  $\|U-U_0\|\leq M\|A-A_0\|^{1/n}$  and $M:=\sum_i \widetilde{M}_i$ with $\widetilde{M}_i$ depending only on $A_0$.
	
	Next, let us consider $T-T_0$. Using the argument similar to above, we conclude that
	$$\|T-T_0\|=\|U^*AU-U_0^*A_{0}U_0\|\leq\widehat{M}\|A-A_0\|,$$
	where $\widehat{M}$ depends only on $A_0$.
	
	Hence, we arrive at the conclusion of one of our main results, i.e. formula~\eqref{eqback} holds true with $K=M+\widehat{M}$.
\end{proof}

%%%%%%%%%%%%%%%%%%%%%%%%%%%%%%%%%%%%%%%%%%%%%%%%%%%%%%%%

\section{Different GK numbers and Failure of the Forward Stability of the Schur Decomposition}

As it turns out the GK numbers of the original and perturbed matrices give us the  information whether the forward stability of the Schur decomposition is impossible. The intuition behind the non-stable case  comes from the following fact.

\begin{proposition}[see \cite{O89}]
 %   Matrices $A$ and $A_0$ have the same GK numbers if and only if 
    %$\inf_{A,A_0}\text{\rm dist}(\text{\rm Inv} A, \text{\rm Inv} A_0)=0$. 
    We have the inequality
    $$\inf\text{\rm dist}(\text{\rm Inv} A, \text{\rm Inv} A_0)>0$$
    where the infimum is taken over all possible pairs of $A,A_0$, having different GK numbers.
\end{proposition}
That is why we got the backward stability result and could not get the general result for forward stability.
%We have seen already that having the same GK numbers means the Schur decomposition of a matrix is at least H\"older stable under small perturbations. Now let us see what happens 
%when the original matrix and its perturbation have different GK numbers.

\begin{lemma}\label{1diffGK}
    Let $A_0\in\mathbb{C}^{n\times n}$ and $A_0=U_0T_0U_0^*$ its fixed Schur decomposition, where the (1,1)-entry of $T_0$ is an eigenvalue $\lambda$ with $\dim\text{\rm Ker}(A_0-\lambda I)\ge2$. There exists $M>0$ such that in any neighborhood of $A_0$,, i.e. $\{A:\|A-A_0\|<\varepsilon\}$ for any $\varepsilon>0$,
\begin{equation}
 \underset{A}{\sup}\ \underset{\begin{array}{c}U,T\\\text{\tiny Schur Form}\\\text{\tiny of $A$}\end{array}}{\inf}\ {\|U-U_0\|+\|T-T_0\|}>M>0,
\end{equation}  where the supremum is taken over all $A\in\mathcal{U}$ not having an eigenvector close to the first column $u_1$ of $U_0$, i.e. we have $\|u_1-v\|>M$ for all $v$, eigenvectors of $A$, and the infimum is taken over all their Schur factorizations.

\end{lemma}

\begin{proof}
     Let $u_1$ be the first column of $U_0$ and $\lambda$ is the corresponding  eigenvalue of $A_0$. By Proposition~\ref{eigstab} we know that there are eigenvalues of $A$ that lie relatively close to $\lambda$ and the difference is equivalent to $\|A-A_0\|^{1/n}$. We will list those  eigenvalues as $\mu_1,\ldots,\mu_l$.
    
     Let $\{v_1,\ldots,v_{l_j}\}$ be a basis of $\text{Ker}(A-\mu_j)$ (for $j=1,\ldots, l$). Moreover, let us denote by $\{u_1,\ldots,u_{k_1(A_0)}\}$ the basis of $\text{Ker}(A_0-\lambda)$, where $u_1$ as before and for each $v_j$ there is $u_s$ ($s\ne 1$) such that $$\|v_j-u_s\|\le C\|A-A_0\|^{1/n}$$ for some positive number $C$. We can assume this by using the backward stability result proven above. Then,
     $$
     \|u_1-v_j\|\ge \|u_1-u_s\|-\|u_s-v_j\|\ge \min_{i\ne1}\|u_1-u_i\|-C\|A-A_0\|^{1/n}.
     $$
     We can always choose $A$ close to $A_0$ so $M:=\frac{\min_{i\ne1}\|u_1-u_i\|}{2}>C\|A-A_0\|^{1/n}$. Denote by $j^*$ the index that minimize the left-hand side of the above inequalities. Therefore, when supremum and infimum is taken under the conditions of this lemma
     $$
 \underset{\widetilde{A}}{\sup}\ \underset{U,T}{\inf}\ {\|U-U_0\|+\|T-T_0\|}\ge \underset{U,T:A=UTU^*}{\inf}\ {\|U-U_0\|}\ge\|u_1-v_{j^*}\|>M.$$
\end{proof}

\begin{remark}
    The assumption that we can find such $A$ in every neighborhood of $A_0$ not having an eigenvector close to the first column of $U_0$ is based on the following fact.  Consider the Jordan form of $A_0=P_0J_0P_0^{-1}$ such that the first two blocks correspond to the eigenvalue $\lambda$ with the second one having $u_1$ as the eigenvector. For any $\varepsilon>0$ take $A=P_0(J_0+J_\varepsilon) P_0^{-1}$, where $J_\varepsilon$ has the only non-zero entry equal to $\frac{\varepsilon}{\|P_0\|\|P_0^{-1}\|}$ on the $(j,j+1)$ spot ($j\times j)$ is the size of the first Jordan block in $J_0$ and the second block corresponds to the eigenvector $u_1$). Hence, $u_1$ is not an eigenvector of $A$ and $\|A-A_0\|=\|P_0J_\varepsilon P_0^{-1}\|\le\varepsilon$. So the set of $A$ that we are taking supremum over in Lemma~\ref{1diffGK} is not empty.
\end{remark}

Now let us show that the statement of Theorem~\ref{GKnonstable} is valid.

\begin{proof}[Proof of Theorem 1.6]
    First, note that having different GK numbers for $A$ and $A_0$ implies that $A_0$ is derogatory, i.e. there is an eigenvalue $\lambda$ of $A_0$ such that $\dim\text{\rm Ker}(A_0-\lambda I)\ge2$. Moreover, $A_0$ and $T_0$ are similar so they have the same GK numbers and $\dim\text{\rm Ker}(T_0-\lambda I)\ge2$ as well. Therefore, we can show the equivalent fact instead, i.e. 
         $$
 \underset{\widetilde{B}}{\sup}\ \underset{U,T}{\inf}\ {\|U-I\|+\|T-T_0\|}\ge \underset{U,T:B=UTU^*}{\inf}\ {\|U-I\|}>M.$$
    If $\lambda$ is the (1,1) entry of $T_0$ the we use Lemma~\ref{1diffGK} to get the desired result. If it is not then we are required to perform an extra step.
    
%    We can decompose $U_0=\widetilde{U}_0\widehat{U}_0$ in such a way that 
    $$
    	T_0={U}_0^*A_0{U}_0=\left[\begin{array}{c|c}
\begin{matrix}
	\star&\star&\cdots&\star&\star\\
0&	\star&\cdots&\star&\star\\
0&	0&\cdots&\star&\star\\
		\vdots&\ddots&\ddots&\ddots&\vdots\\	
		0&0&\cdots&0&\star
	\end{matrix}&\begin{matrix}
	\star&\cdots&\star\\
	\star&\cdots&\star\\
	\star&\cdots&\star\\
		\star&\cdots&\star\\
			\star&\cdots&\star\\
				\star&\cdots&\star
	\end{matrix}\\ \hline\\
\text{\huge 0}&\text{\huge ${T}_{1}$}
	\end{array}\right],
    $$
    where the first say $j$ rows do not have $\lambda$ on its main diagonal and the next row of $T_0$ is the first time we meet $\lambda$. 
    In addition, the (1,1)-entry of ${T}_1$ is $\lambda$. According to Lemma~\ref{reducingJordan}, it means that we have the same GK numbers related to $\lambda$ for $T_1$ as for $A_0$. Recall that the truncation was using eigenvectors not corresponding to $\lambda$, so $T_1$ will have the same number and length of Jordan chains for $\lambda$ as $A_0$ has. Thus, $\dim\text{\rm Ker}(T_1-\lambda I)\ge2$. Now, we use Lemma~\ref{1diffGK} to get the result for $\widetilde{A}_0$, $V_0$, and $\widetilde{T}_0$ by constructing $B_1$ (note that $B$ is not upper triangular, since $e_1$ is not its eigenvector) for $T_1$ as described by the lemma. Then define
      $$
    	A={U}_0\left[\begin{array}{c|c}
\begin{matrix}
	\star&\star&\cdots&\star&\star\\
0&	\star&\cdots&\star&\star\\
0&	0&\cdots&\star&\star\\
		\vdots&\ddots&\ddots&\ddots&\vdots\\	
		0&0&\cdots&0&\star
	\end{matrix}&\begin{matrix}
	\star&\cdots&\star\\
	\star&\cdots&\star\\
	\star&\cdots&\star\\
		\star&\cdots&\star\\
			\star&\cdots&\star\\
				\star&\cdots&\star
	\end{matrix}\\ \hline\\
\text{\huge 0}&\text{\huge $B_1$}
	\end{array}\right]{U}_0^*,
    $$  where the first $j$ rows marked with stars coincide with $T_0$. By the construction, we can see that \eqref{nonstable} holds true and hence finishing the proof of Theorem~\ref{GKnonstable}.
\end{proof}

To summarize, we have showed that the Schur decomposition is backward stable and why it fails to be forward stable.
%%%%%%%%%%%%%%%%%%%%%%%%%%%%%%%%%%%%%%%%%%%%%%%%%%%%%%%

%%  \bibliography{<your bibdatabase>}

\end{document}